\documentclass[a4paper,10pt,reqno]{amsart}
\usepackage[T1]{fontenc}
\usepackage[utf8]{inputenc}
\usepackage{amssymb}
\usepackage{amsmath}
\usepackage{makecell}
\usepackage{multirow}
\usepackage{hhline}
\usepackage{amsthm}
\usepackage{amsfonts}
\usepackage{mathtools}
\usepackage{todonotes}
\usepackage{comment}
\usepackage{mathrsfs}
\usepackage{enumitem}
\usepackage[all,cmtip]{xy}
\usepackage{ctable}
\usepackage[hyphens]{url}
\usepackage{hyperref}
\usepackage[hyphenbreaks]{breakurl}
\usepackage{cleveref}
\usepackage{lipsum}
\usepackage{subcaption}
\usepackage{tabularx}
\usepackage{soul}
\usepackage{algorithm}

\usepackage[noend]{algpseudocode}
\usepackage[british]{babel}
\usepackage[margin=1.2in]{geometry}

\usepackage[british]{babel}
\usepackage{eqparbox}

\usepackage{pbox}

\newcommand{\field}[1]{\mathbb{#1}}  
\newcommand{\Q}{\field{Q}} 
\newcommand{\R}{\field{R}} 
\newcommand{\Z}{\field{Z}} 
\newcommand{\F}{\field{F}} 

\newcommand{\Mod}[1]{\ (\mathrm{mod}\ #1)}

\DeclareMathOperator{\Frob}{Frob}

\DeclareMathOperator{\Gal}{Gal}

\DeclareMathOperator{\ord}{ord}
\DeclareMathOperator{\BSD}{BSD}

\DeclareMathOperator{\Reg}{\textup{Reg}}

\DeclareFontFamily{U}{wncy}{}
\DeclareFontShape{U}{wncy}{m}{n}{<->wncyr10}{}
\DeclareSymbolFont{mcy}{U}{wncy}{m}{n}
\DeclareMathSymbol{\Sha}{\mathord}{mcy}{"58}

\newtheorem{lemma}{Lemma}
\newtheorem{theorem}[lemma]{Theorem}
\newtheorem{proposition}[lemma]{Proposition}

\theoremstyle{definition}
\newtheorem{definition}[lemma]{Definition}

\newtheorem{remark}[lemma]{Remark}

\newtheorem*{ack}{Acknowledgements}

\numberwithin{lemma}{section}
\numberwithin{equation}{section} 
\numberwithin{figure}{section}

\title[Infinite families of elliptic curves satisfying BSD]{On the identification of elliptic curves that admit infinitely many twists satisfying the Birch--Swinnerton-Dyer conjecture}

\author[B.S. Banwait]{Barinder S. Banwait}
\address{Barinder S. Banwait \\
London\\
UK
}
\email{barinder.s.banwait@gmail.com}

\author[X. Huang]{Xiaoyu Huang}
\address{Xiaoyu Huang\\
Department of Mathematics \\ 
Temple University\\
Wachman Hall\\
Philadelphia, PA 19122\\
USA
}
\email{xiaoyu.huang@temple.edu}

\date{}

\makeatletter
\providecommand\@dotsep{5}
\renewcommand{\listoftodos}[1][\@todonotes@todolistname]{%
  \@starttoc{tdo}{#1}}
\makeatother


\begin{document}

\maketitle

\begin{abstract}
Recent work of Burungale--Skinner--Tian--Wan established the first infinite families of quadratic twists of non-CM elliptic curves over $\mathbb{Q}$ for which the strong Birch--Swinnerton-Dyer (BSD) conjecture holds. Building on their results, we encode the required hypotheses into an explicit algorithm and apply it to the database of elliptic curves in the $L$-functions and Modular Forms Database (LMFDB), unconditionally identifying all elliptic curves $E$ of conductor at most $500{,}000$ that admit infinitely many quadratic twists satisfying the strong BSD conjecture. Our computations provide certain numerical evidence for a conjecture of Radziwiłł and Soundararajan predicting Gaussian behavior in the analytic order of the Shafarevich--Tate group, while also observing a positive bias within the BSD-satisfying subfamily.
\end{abstract}

\section{Introduction}

Significant progress has been made toward the Birch and Swinnerton-Dyer (BSD) conjecture from a statistical perspective. Through the work of Bhargava and Shankar on average ranks \cite{bhargava2015binary}, combined with the Iwasawa-theoretic results of Skinner--Urban \cite{skinner2014iwasawa}, Zhang \cite{zhang2014selmer}, and others, it is now established that at least $66\%$ of elliptic curves satisfy the rank part of the conjecture \cite{bhargava2014majority}. These results, however, are probabilistic in nature; they assure us that the conjecture holds frequently, but they do not necessarily tell us which specific curves satisfy the full BSD conjecture.

A complementary perspective is the algorithmic verification of the conjecture for specific curves up to a given conductor bound. As a result of an accumulation of works involving several people across many years \cite{grigorov2009computational, miller2011proving, creutz2012second, miller2013explicit, lawson2015vanishing}, it is now known that the full BSD conjecture holds for all elliptic curves over $\mathbb{Q}$ with analytic rank $0$ or $1$ of conductor up to $5{,}000$. At the most basic level, this is achieved by computing the order of the Shafarevich--Tate group $\Sha(E)$ and showing it matches the other terms in the BSD conjectural formula (which are neatly packaged into a quantity $\Sha(E)_{an}$ called the \emph{analytic order of $\Sha(E)$}). There are $31{,}073$ elliptic curves with conductor less than $5{,}000$; out of these, only $691$ have analytic rank strictly greater than $1$, meaning that the full BSD conjecture has been verified for $\approx97.78\%$ of elliptic curves of conductor less than $5{,}000$. We refer to the end of Section 5 of \cite{lawson2015vanishing} for an account of how the multiple aforementioned works fit together to prove the stated result. 

In this paper, we combine the algorithmic rigor of the second approach with the infinite scope of the first by identifying specific elliptic curves up to a conductor bound that admit \emph{infinitely many} quadratic twists satisfying the full BSD conjecture. Until recently, such families were known primarily for curves with Complex Multiplication (CM), thanks to the work of Coates--Wiles \cite{coates1977conjecture} and Rubin \cite{rubin1992p}.

For non-CM curves, the landscape has expanded significantly with the recent work of Burungale--\allowbreak Skinner--\allowbreak Tian--\allowbreak Wan \cite{burungale2024zeta}. By establishing the \(p\)-part of the BSD conjecture at supersingular primes, together with earlier results in \cite{skinner2014iwasawa, skinner2016multiplicative, cai20202, zhai2016non} they proved the existence of infinite families of quadratic twists of non-CM elliptic curves satisfying the full BSD conjecture. In their work, they provided a list of such curves with conductor up to \(150\). While this list is given explicitly, details of the underlying algorithmic procedure for identifying these curves, as well as an implementation for verifying the requisite local conditions, are not made explicit.
 
Our first contribution is to translate the theoretical conditions of Burungale--Skinner--Tian--Wan into explicit, executable algorithms. We implement these checks and run them on the entire database of elliptic curves over $\Q$ in the $L$-functions and Modular Forms Database (LMFDB) \cite{lmfdb}. This allows us to extend the list of known examples, identifying all elliptic curves of conductor $N < 500{,}000$ that unconditionally admit infinite families of BSD-satisfying twists.

\begin{theorem}
Let $\mathcal{C}$ be the set of elliptic curves listed in \cite{banwait2024quadtwist}. For every $E \in \mathcal{C}$, there exists an explicit infinite family of quadratic twists $\{E_d\}$ such that each $E_d$ satisfies the full Birch and Swinnerton-Dyer conjecture.
\end{theorem}

The work \cite{banwait2024quadtwist} identifies $36{,}687$ semistable elliptic curves admitting infinitely many quadratic twists satisfying the BSD conjecture, accounting for approximately $20.5\%$ of optimal semistable analytic-rank-0 elliptic curves and $1.20\%$ of all elliptic curves with conductor \(N<500{,}000\). 
We stress that unlike the probabilistic results, the BSD conjecture holds \emph{unconditionally} for these families. This provides us with a unique dataset: a large collection of elliptic curves where the order of the Shafarevich--Tate group $|\Sha(E_d)|$ is known unconditionally.

Our second contribution uses this dataset to study the statistical distribution of $\Sha(E_d)$. A conjecture of Radziwiłł and Soundararajan \cite{radziwill2015moments}, motivated by random matrix theory, predicts that for a fixed elliptic curve $E$, the values of $\log|\Sha(E_d) - \frac{1}{2}\log|d|$ (suitably normalized) should follow a standard Gaussian distribution as $d$ varies. Previously, this conjecture could only be tested numerically by assuming the validity of BSD to interpret analytic $L$-values as actual group orders \cite{dabrowski2016behaviour}.

Because our dataset consists of families for which the full BSD conjecture is now unconditionally established, we are able to study the true order of $\Sha(E_d)$ without relying on this assumption. We structure our analysis in two steps. First, to establish a baseline, we look at generic twists; assuming BSD for these twists, we recover the predicted Gaussian distribution. Second, we examine our specific, unconditionally BSD-verified families. Here, we observe a significant deviation: the true distribution for these families exhibits a positive bias compared to the generic Radziwiłł--Soundararajan prediction. This deviation is not unexpected: the BSD-verified twists form a constrained arithmetic subfamily, cut out by local congruence and splitting conditions on twist parameter $d$, rather than an unbiased sample of all quadratic twists.

Our work is supported by a code repository that is available here:
    \begin{center}
    \url{https://github.com/cocoxhuang/ants_xvii}
    \end{center}
We consider this repository a project that is built on top of the LMFDB, which we forked and are making use of in our work, all of which is contained in the \texttt{ants\_xvii} folder. See the \texttt{README.md} at the top level of the repository for instructions on how to use the code to verify the claims made in this paper. Our algorithms are written in SageMath \cite{sagemath}. \textbf{All file paths in this paper will be relative to the \texttt{ants\_xvii} directory of the above repository.}

\textbf{Computing Environment.} All computations reported here were carried out on a single laptop: an Apple MacBook Pro with an Apple~M1~Pro processor (eight CPU cores) and 16~GB of unified memory, running macOS~26.4.1. The implementation uses SageMath~10.7, which bundles Python~3.13.3. Algorithm~1, executed against the full LMFDB elliptic-curve database at $N < 500{,}000$ (3{,}064{,}705 curves), runs to completion in 12~minutes and 26~seconds on this machine.

\ack{
    We thank the anonymous referees for their valuable comments and for suggesting recent work that led to improvements in Algorithm~1 as well as the exposition in general. We are grateful to Shuai Zhai for answering our questions regarding the curves listed in \cite[Example~1]{burungale2024zeta}. We also thank Kannan Soundararajan for proposing the question of studying the empirical distribution of \( \lvert \Sha(E_d) \rvert \). We further thank Ashay Burungale, Chris Skinner, Ye Tian, and Xin Wan for answering our questions regarding their results. We would also like to thank Jaclyn Lang for helpful discussions.
}

\section{Quadratic Twist Families That Satisfy the Full BSD Conjecture}\label{sec : p-BSD}


The main contributions of this section are \Cref{algo: base_curve} and \Cref{algo: twists}, which respectively (a) identify elliptic curves admitting infinitely many quadratic twists satisfying full BSD, and (b) determine these twists up to a bound on the twisting parameter. These algorithms are built upon many theoretical advances in the Iwasawa theory of elliptic curves, spanning the last $10$ years, that establish the \emph{$p$-part of BSD} (to be defined in \Cref{ssec:prelims}) under various conditions. After discussing preliminaries, we discuss the $p$-part of BSD for odd primes $p$ in \Cref{ssec:p-part-BSD}, and the $2$-part in \Cref{ssec:2-part-BSD}. There is a certain set in \Cref{ssec:2-part-BSD} whose non-emptiness we need to understand for our algorithm to work; this is discussed in \Cref{ssec:S_non_empty_conds}. We put the $p$-parts and $2$-part together in \Cref{ssec:parts-together} to obtain the mathematical statement \Cref{thm:main_alg}. This result forms the basis of the algorithms, which are declared in pseudocode in \Cref{ssec:algs}. Finally in \Cref{ssec:bsd_curves_cond_150} we discuss the elliptic curves satisfying these conditions up to conductor $150$, and compare our list with the corresponding list in \cite[Example 1]{burungale2024zeta}.

\subsection{Notation and preliminaries}\label{ssec:prelims}

In the sequel, we will use the notation $\BSD(E)$ to mean that the strong Birch--Swinnerton-Dyer conjecture has been established. Following Miller \cite{miller2011proving}, we make the following definitions.

\begin{definition}
We define the \textbf{analytic order of $\Sha(E)$} as follows.
    \begin{equation}\label{eq:strong_bsd}
        \#\Sha(E)_{\mathrm{an}} := \frac{L^{(r)}(E,1)}{r!} \cdot
        \frac{|E(\Q)_\mathrm{tors}|^2}{\Omega(E) \cdot \Reg(E) \cdot
        \prod_{p} c_p},
    \end{equation}
where $r = r_{\mathrm{an}}(E)$, so that the leading factor $L^{(r)}(E,1)/r!$ is the
leading nonzero Taylor coefficient of $L(E,s)$ at $s=1$.
\end{definition}

\begin{definition}\label{def:bsd_archimedean}
We denote by $\BSD(E/\Q, \infty)$ the assertion that the real number
$\#\Sha(E)_{\mathrm{an}}$ is positive.
\end{definition}

\begin{definition}
    We denote by $\BSD(E/\Q, p)$ the following assertions.
    \begin{enumerate}
        \item The rank $r$ of $E(\Q)$ is equal to the analytic rank $r_{\mathrm{an}}(E)$.
        \item The $p$-primary part $\Sha(E)(p)$ of the Shafarevich--Tate group is finite.
        \item The real number $\#\Sha(E)_{\mathrm{an}}$ is rational.
        \item The conjectural formula holds at $p$; that is, \[ \ord_p(\#\Sha(E)_{\mathrm{an}}) = \ord_p(\#\Sha(E)(p)).\]
    \end{enumerate}
\end{definition}

Write $M_\Q = \{\,\text{primes of } \Q\,\} \cup \{\infty\}$ for the set of places of $\Q$. The following is immediate from the definitions, since two positive rational numbers are equal if and only if they have the same $v$-adic valuation for every place $v$.

\begin{proposition}\label{fact:bsd_iff_p}
$\BSD(E)$ holds if and only if $\BSD(E/\Q, v)$ holds for all $v \in M_\Q$.
\end{proposition}


We now specialise to the case $r_{\mathrm{an}}(E) \le 1$, which is the only case
occurring in this paper.

\begin{theorem}[Gross--Zagier, Kolyvagin, Kohnen--Zagier]\label{thm:rank01_reduction}
Suppose $r_{\mathrm{an}}(E) \le 1$. Then assertions (1), (2), (3) of
$\BSD(E/\Q,p)$ hold for every prime $p$, and $\BSD(E/\Q,\infty)$ holds.
Consequently,
\[
    \BSD(E) \ \Longleftrightarrow \
    \ord_p(\#\Sha(E)_{\mathrm{an}}) = \ord_p(\#\Sha(E)(p))
    \quad \text{for all primes } p.
\]
\end{theorem}

\begin{proof}
The hypothesis $r_{\mathrm{an}}(E) \le 1$ implies $r = r_{\mathrm{an}}(E)$
(assertion (1)) and the finiteness of $\Sha(E)$ (assertion (2)), by
Gross--Zagier \cite{GrossZagier1986} and Kolyvagin \cite{Kolyvagin1988}. The
rationality of $\#\Sha(E)_{\mathrm{an}}$ (assertion (3)) follows from the
rationality of $L(E,1)/\Omega(E)$ when $r_{\mathrm{an}}=0$ \cite{Manin1972} and
from the Gross--Zagier formula when $r_{\mathrm{an}}=1$. For $\BSD(E/\Q,\infty)$
we must show $L^{(r)}(E,1)/r! > 0$: if $r_{\mathrm{an}}=0$, the non-negativity of
the central value $L(E,1) \ge 0$ \cite{kohnen1981values} together with
$L(E,1) \ne 0$ gives $L(E,1) > 0$; if $r_{\mathrm{an}}=1$, the Gross--Zagier
formula expresses $L'(E,1)$ as a positive multiple of the N\'eron--Tate height of
a non-torsion Heegner point, so $L'(E,1) > 0$.

By Proposition~\ref{fact:bsd_iff_p}, and as assertions (1)--(3) and
$\BSD(E/\Q,\infty)$ have now been verified, $\BSD(E)$ is equivalent to assertion
(4) holding at every finite prime. Concretely, writing
$\Q^\times \cong \{\pm 1\} \times \bigoplus_p p^{\Z}$, the finite valuations
$(\ord_p q)_p$ of a nonzero rational $q$ fix only its absolute value
$|q| = \prod_p p^{\ord_p q}$ (the product formula); thus the equalities
$\ord_p(\#\Sha(E)_{\mathrm{an}}) = \ord_p(\#\Sha(E)(p))$ pin down
$\#\Sha(E)_{\mathrm{an}}$ up to sign, and the positivity $\BSD(E/\Q,\infty)$
supplies the missing sign at the real place.
\end{proof}

\begin{definition}
Let $E/\Q$ be an elliptic curve of conductor $N$. Let
\(
\phi \colon X_0(N)\to E
\)
be a modular parametrization sending the cusp at infinity to the origin of $E$. We say that $E$ is \emph{optimal} if $\phi$ does not factor through any other elliptic curve in the isogeny class of $E$. If $E$ is optimal and $\omega$ is the N\'eron differential on a global minimal Weierstrass equation for $E$, then the unique integer $\nu_E\in\Q^\times$ satisfying
\[
\phi^*(\omega)=\nu_E f(\tau)\,d\tau,
\]
where $f$ is the normalized newform associated to $E$, is called the \emph{Manin constant} of $E$.
\end{definition}

See \cite{agashe2006manin} for more on the Manin constant. It is conjectured to be $1$ in all cases.

\begin{definition}
    Let $E/\Q$ be an elliptic curve of analytic rank $0$. We define the \emph{algebraic $L$-value} as \[
L^{(\mathrm{alg})}(E,1) := \frac{L(E,1)}{\Omega(E)}.
\]
\end{definition}

\subsection{$p$-part of BSD for $p$ odd}\label{ssec:p-part-BSD}

Building on recent advances \cite[Theorem 2]{skinner2014iwasawa}, \cite[Corollary 10.2]{burungale2024zeta}, \cite[Theorem C]{skinner2016multiplicative}, and \cite[Theorem 9.21]{burungale2024zeta}, it is now known that a large class of elliptic curves $E$ satisfy $\BSD(E,p)$ for odd primes $p$. The case relevant to our work is when the quadratic twists have rank zero.

We begin by isolating the following step to be used in the sequel. It is a slight strengthening of Ribet's level-lowering argument from \cite{ribet1990modular}.

\begin{lemma}\label{lem:souped_level_lowering}
Let $E/\Q$ be a semistable elliptic curve of conductor $N$, and let $p$ be an odd
prime of good ordinary reduction for $E$ such that $\overline{\rho}_{E,p}$ is
irreducible. Then there exists an \emph{odd} prime $q \neq p$ with $q \parallel N$
at which $\overline{\rho}_{E,p}$ is ramified.
\end{lemma}

\begin{proof}
Since $E$ is semistable, $N$ is squarefree, and at each prime $\ell \mid N$ the
curve $E$ has multiplicative reduction; thus $\overline{\rho}_{E,p}$ is ramified at
such an $\ell \neq p$ if and only if $p \nmid \mathrm{ord}_\ell(\Delta_E)$, in which
case its conductor exponent at $\ell$ is $1$, and otherwise $\overline{\rho}_{E,p}$
is unramified at $\ell$. At $p$ itself, good reduction gives prime-to-$p$ conductor.
Consequently the (prime-to-$p$) Serre conductor $N(\overline{\rho}_{E,p})$ is the
squarefree product of those primes $\ell \parallel N$, $\ell \neq p$, at which
$\overline{\rho}_{E,p}$ ramifies; in particular $N(\overline{\rho}_{E,p}) \mid N$ and
is squarefree and prime to $p$. Moreover, as $p$ is a prime of good \emph{ordinary}
reduction, the Serre weight is $k(\overline{\rho}_{E,p}) = 2$.

Suppose, for contradiction, that $\overline{\rho}_{E,p}$ is unramified at every
\emph{odd} prime $q \parallel N$ with $q \neq p$. Then the only prime that can divide
$N(\overline{\rho}_{E,p})$ is $2$, so $N(\overline{\rho}_{E,p}) \in \{1, 2\}$. Since
$\overline{\rho}_{E,p}$ is irreducible with $p$ odd, Ribet's level-lowering theorem
\cite{ribet1990modular} (together with the modularity of $E$) produces a normalized
eigenform of weight $k(\overline{\rho}_{E,p}) = 2$ on $\Gamma_0(N(\overline{\rho}_{E,p}))$,
hence on $\Gamma_0(1)$ or $\Gamma_0(2)$, giving rise to $\overline{\rho}_{E,p}$. But
the modular curves $X_0(1)$ and $X_0(2)$ both have genus $0$, so
$S_2(\Gamma_0(1)) = S_2(\Gamma_0(2)) = 0$; there are no such eigenforms, yielding the desired contradiction.
\end{proof}

\begin{proposition}\label{thm:p_part_bsd}
Let $p$ be an odd prime, $E$ a semistable elliptic curve over $\Q$ with conductor $N$, and $E_d$ the quadratic twist of $E$ by the character associated to the quadratic extension $\mathbb{Q}(\sqrt{d})/\mathbb{Q}$, where $(d,N) = 1$. Write $D$ for the discriminant of this quadratic field, and write $N_{E_d}$ for the conductor of $E_d$. Suppose $L(E_d, 1) \ne 0$. Then $p$ either divides $d$, or else is good ordinary, good supersingular, or multiplicative for $E_d$. In each of these four cases, if the following conditions hold, then $\BSD(E_d,p)$ holds.

\begin{enumerate}
    \item {\label{res:additive}} If $p \mid d$:
    \begin{enumerate}
        \item $p \geq 5$ and $p$ is a good ordinary prime for $E$;
        \item $\overline{\rho}_{E,p}$ is irreducible.
    \end{enumerate}

    \item {\label{res:good}}
    If $p \nmid d$ and is good ordinary: no further conditions.

    \item {\label{res:multi}}
    If $p \nmid d$ and is multiplicative:
    \begin{enumerate}
        \item $\overline{\rho}_{E_d,p}$ is irreducible;
        \item there is a prime $q \ne p$ such that $q \parallel N_{E_d}$ and $\overline{\rho}_{E_d,p}$ is ramified at $q$. 
    \end{enumerate}

    \item {\label{res:sing}}
    If $p \nmid d$ and $p$ is supersingular:
        \begin{enumerate}
            \item $(D,N) = 1$;
            \item if $p = 3$, then $a_3(E) = 0$.
        \end{enumerate}
\end{enumerate}
\end{proposition}

\begin{proof}
    The four cases correspond to the different reduction types for $E_d$ at $p$, with $p \mid d$ being equivalent to $E_d$ having additive reduction at $p$ (because the primes dividing $d$ divide the conductor of $E_d$ to order at least $2$).
    \begin{enumerate}
        \item Since $p \mid d$, we have $p \mid D$, so this is the additive twist case of
        \cite[Theorem~9.21(c)]{burungale2024zeta}, applied with $K = \Q(\sqrt{d})$: conditions
        (a) and (b) supply the hypotheses that $p \geq 5$ is ordinary for $E$ with
        $\overline{\rho}_{E,p}$ irreducible (their $p \nmid 6N$ and $(\mathrm{irr}_\Q)$). For
        their condition $(\mathrm{ram}_K)$, note that (a), (b), and the semistability of $E$
        are exactly the hypotheses of \Cref{lem:souped_level_lowering}, which furnishes an odd
        prime $q \neq p$ with $q \parallel N$ at which $\overline{\rho}_{E,p}$ is ramified;
        since $q$ is odd and $q \nmid d$, it does not divide $D$, giving
        $(\mathrm{ram}_K)$. The passage from the equality of Iwasawa ideals in \cite[Theorem~9.21(c)]{burungale2024zeta} to the $p$-part of the BSD formula for $E_d$ is Kato's rank-zero descent (cf.\ the $r=0$ ordinary case in the proof of \cite[Corollary~10.2]{burungale2024zeta}, following \cite[\S14.20]{kato2004p}).

        \item This case was first dealt with by Skinner--Urban \cite[Theorem 3.33, Theorem 3.35]{skinner2014iwasawa} under the following additional conditions:
        \begin{itemize}
            \item surjectivity of the mod-$p$ representation $\bar{\rho}_{E_d,p}$;
            \item the \emph{ramified prime} condition: there exists a prime $q \mid\mid N_{E_d}$, $q \neq p$, such that $\bar{\rho}_{E_d,p}$ is ramified at $q$.
        \end{itemize}
        Subsequently, in \cite[Theorem C]{skinner2016multiplicative}, the surjectivity condition was weakened to the representation being irreducible.

        We now show that, when $\bar{\rho}_{E_d,p}$ is irreducible, the ramified prime condition follows from semistability of $E$. Since $p$ is odd, $\overline{\rho}_{E_d,p} \cong \overline{\rho}_{E,p} \otimes \chi_d$,
        where $\chi_d$ is the quadratic character of $\Q(\sqrt{d})/\Q$; twisting by a
        character preserves irreducibility, so $\overline{\rho}_{E,p}$ is irreducible as well.
        As $p \nmid d$ and $(d,N) = 1$, the prime $p$ is one of good reduction for $E$, and it
        is good ordinary for $E$ since it is so for $E_d$ (ordinarity being preserved under
        quadratic twist away from $p$). Thus $E$, $p$ satisfy the hypotheses of
        \Cref{lem:souped_level_lowering}, which furnishes an odd prime $q \neq p$ with
        $q \parallel N$ at which $\overline{\rho}_{E,p}$ is ramified.
        
        We claim this same $q$ witnesses the ramified prime condition for $E_d$. Indeed, $q$
        is odd and $q \mid N$, while $(d,N) = 1$ forces $q \nmid d$; hence $q$ does not divide $D$, so $\chi_d$ is unramified at $q$.
        Therefore $\overline{\rho}_{E_d,p}\vert_{I_q} = \overline{\rho}_{E,p}\vert_{I_q}$ is
        ramified, and the conductor exponent of $E_d$ at $q$ equals that of $E$, namely $1$;
        that is, $q \parallel N_{E_d}$ with $\overline{\rho}_{E_d,p}$ ramified at $q$. This is
        precisely the ramified prime condition for $E_d$, and hence \cite[Theorem~C]{skinner2016multiplicative} gives $\BSD(E_d,p)$ in the irreducible case.
        
        The reducible case was addressed for all odd primes by \cite[Theorem D]{castella2025mazur} (building on \cite[Theorem F]{castella2022anticyclotomic}) up to a certain condition on the isogeny character denoted $\phi$ in \emph{loc. cit.}. This condition was removed by \cite[Theorem C]{keller2024anticyclotomic}. Thus, we have $\BSD(E_d,p)$ in both cases (irreducible and reducible).
        \item This is \cite[Theorem C]{skinner2016multiplicative}.
        \item Here $p \nmid d$ is supersingular for $E_d$, hence also for $E$, and
        $a_p(E_d) = \chi_d(p)\,a_p(E) = \pm\,a_p(E)$. We verify the hypotheses of
        \cite[Corollary~10.2]{burungale2024zeta} for $E$, applied with the quadratic field
        $K = \Q(\sqrt{d})$. As $E$ is semistable its conductor $N$ is squarefree, and
        $p \nmid 2N$ since $p$ is odd and of good reduction for $E$. The twisting field has
        discriminant $D$, which is coprime to $Np$ by assumption (a). Finally we check
        $a_p(E) = 0$. If $p > 3$, supersingularity forces $a_p(E) = 0$ by the Hasse bound; and if $p = 3$ then $a_3(E) = 0$ by condition (b).
        
        Thus $E$ satisfies the hypotheses of \cite[Corollary~10.2]{burungale2024zeta}, which
        yields $\BSD(E,p)$; the final clause of that corollary, valid for quadratic twists by a field of discriminant coprime to $Np$, then gives $\BSD(E_d, p)$.\qedhere
    \end{enumerate}
\end{proof}

\begin{remark}
    The semistable condition in the above result is a strong restriction, and is needed for \Cref{res:sing}, \Cref{res:additive},  and for \Cref{res:good} when $\overline{\rho}_{E_d,p}$ is irreducible.

    \Cref{res:sing} relies on \cite[Corollary~10.2]{burungale2024zeta},
    whose proof rests on the resolution of Kobayashi's main conjecture
    \cite[Theorem~10.1]{burungale2024zeta} and which carries the squarefree condition on
    the conductor.
    
    \Cref{res:additive} uses semistability only to supply the \emph{ramified prime} condition via \Cref{lem:souped_level_lowering}; without it, one adds that condition as a hypothesis.
    
    \Cref{res:multi} holds for non-semistable curves since the underlying result \cite[Theorem~C]{skinner2016multiplicative} applies with no restriction on the conductor.
    
    For \Cref{res:good}, we first consider the case of reducible residual representation.
    If $\bar\rho_{E_d,p}$ is reducible, the result holds for non-semistable curves for
    every odd $p$, by \cite[Theorem~D]{castella2025mazur} (with the condition on $\phi$
    removed in \cite[Theorem~C]{keller2024anticyclotomic}), just as in the above proof.
    When $\bar\rho_{E_d,p}$ is irreducible, however, our argument uses the semistability
    of $E$ to produce the ramified prime via \Cref{lem:souped_level_lowering}, which is
    unavailable for a non-semistable base curve. In that setting one may verify the ramified prime condition directly; or, for $p \ge 5$, there is work of Burungale--Castella--Skinner
    \cite[Corollary~1.3.1]{burungale2025base}, which establishes the irreducible case for
    non-semistable curves at the expense of further hypotheses: the curve being non-CM,
    and a condition on the image of $\bar\rho_{E_d,p}$.
\end{remark}

\begin{remark}
    It is important to point out the parts of the above Proposition that depend on unpublished results in the literature. These are Items \ref{res:additive} and \ref{res:sing}, which depend on \cite{burungale2024zeta} and \cite{castella2018iwasawa}; and the reducible case of \ref{res:good}, which depends on \cite{keller2024anticyclotomic}.

    We briefly mention some other related unpublished works that could be useful for future possible algorithmic extensions.

    \begin{itemize}
        \item The preprint \cite{fouquet2021iwasawa} deals with arbitrary reduction type under various assumptions on the mod-$p$ Galois representation attached to $E$, as listed in Theorem 1.7 of \emph{loc. cit.}. It is not immediately apparent how to algorithmically verify these conditions.
        \item The paper \cite{yan2026main} is subsequent to \cite{burungale2024zeta}, adopts a similar approach, and has already appeared in print. Their Corollary 1.4 deals with the good ordinary case for both analytic rank $0$ and $1$ under an assumption labeled `(Im)' appearing in their Theorem 1.2. For our purposes of analytic rank 0, since we don't have \emph{any} additional conditions here, Corollary 1.4 of this paper does not provide an additional input. However, this could be useful for possible extensions of our algorithmic work to analytic rank $1$, although, as with \cite{fouquet2021iwasawa}, it is not clear how to algorithmically verify the condition `(Im)'.
    \end{itemize}
\end{remark}

Since many of the results used in the proof of \Cref{thm:p_part_bsd} apply to a general semistable curve $E$, we record what can be said about $\BSD(E,p)$ directly.

\begin{proposition}\label{thm:p_part_bsd_base}
Let $p$ be an odd prime, $E$ a semistable elliptic curve over $\Q$ with
conductor $N$, and suppose $L(E,1)\neq 0$. Then $p$ is good ordinary, good
supersingular, or multiplicative for $E$, and in each case $\BSD(E,p)$ holds under
the following conditions.
\begin{enumerate}
    \item\label{cond:p_ordinary} If $p$ is good ordinary: no further conditions.
    \item\label{cond:p_mult} If $p$ is multiplicative:
    \begin{enumerate}
        \item $\overline{\rho}_{E,p}$ is irreducible;
        \item there is a prime $q \neq p$ with $q \parallel N$ at which
              $\overline{\rho}_{E,p}$ is ramified.
    \end{enumerate}
    \item\label{cond:p_ss} If $p$ is supersingular: no further conditions.
\end{enumerate}
\end{proposition}

\begin{proof}
Since $E$ is semistable it has no additive primes, leaving the three cases above.
Each is the $d=1$ specialization of the corresponding case of \Cref{thm:p_part_bsd},
with $E$ in place of $E_d$, and we only indicate the (minor) differences.

For \Cref{cond:p_ordinary}, the argument of \Cref{res:good} applies verbatim: when
$\overline{\rho}_{E,p}$ is irreducible, \Cref{lem:souped_level_lowering} supplies a
prime $q\parallel N$ at which it is ramified, whence \cite[Theorem~C]{skinner2016multiplicative}
gives $\BSD(E,p)$; the reducible case is \cite[Theorem~D]{castella2025mazur}
(with the condition on $\phi$ removed in \cite[Theorem~C]{keller2024anticyclotomic}).
\Cref{cond:p_mult} is again \cite[Theorem~C]{skinner2016multiplicative}.

For \Cref{cond:p_ss}, if $p>3$, or $p=3$ with $a_3(E)=0$, this is
\cite[Theorem~1.5]{burungale2024zeta} applied to $E$. The remaining case $p=3$,
$a_3(E)=\pm3$ is \cite[Theorem~C]{castella2018iwasawa}, whose hypothesis that $E$
admit no rational $p$-power isogeny is automatic: a rational $p$-isogeny would give a
$G_{\Q_p}$-stable line in $E[p]$, contradicting the irreducibility of
$\overline{\rho}_{E,p}\vert_{G_{\Q_p}}$ at a supersingular prime (tame inertia acting
through the fundamental characters of level $2$).
\end{proof}

\Cref{thm:p_part_bsd} is used to establish the following result, which is a sharpened version of \cite[Theorem 10.12]{burungale2024zeta}.

\begin{theorem}\label{thm:main}
    Let $E$ be a semistable elliptic curve over $\Q$ with conductor $N$, and $d > 1$ a squarefree integer. Let $E_d$ denote the quadratic twist of $E$ by the character associated to the quadratic extension $\Q(\sqrt{d})/\Q$ of discriminant $D$. Suppose the following conditions hold.
    \begin{enumerate}[label=(\roman*)]
        \item\label{cond:coprime_fund_disc} $(D,3N) = 1$. 
        \item\label{list-item:analytic_rank_zero} $L(E_d, 1) \neq 0$.
        \item\label{list-item:bsd_at_2_for_twist} $\BSD(E_d, 2)$ is true.
        \item\label{cond:supersingular} $a_3(E) \in \left\{-2,  -1, 0, 1, 2\right\}$.
        \item\label{cond:abs_irred} For $p \in \left\{3,5,7\right\}$, $E$ does not admit a rational $p$-isogeny.
        \item\label{cond:ramified_prime} For any prime $p | N$, there exists a multiplicative prime $q \neq p$ of $E$ at which $E[p]$ is ramified.
        \item\label{cond:ord_red} E has ordinary reduction at prime divisors of $d$.
    \end{enumerate}
    Then $\BSD(E_d)$ is true.
\end{theorem}

\begin{proof}
    From \Cref{thm:rank01_reduction} applied to $E_d$ (which is permissible because of assumption~\ref{list-item:analytic_rank_zero}), we need to show that $\BSD(E_d, p)$ holds for all $p$. Condition~\ref{list-item:bsd_at_2_for_twist} treats the prime $2$, so we are reduced to showing $\BSD(E_d, p)$ for odd primes $p$.

    First suppose that $p \mid d$. From Case~\ref{res:additive} of \Cref{thm:p_part_bsd}, we need to show $p \neq 3$, $p$ is a good ordinary prime for $E$, and $\bar{\rho}_{E,p}$ is irreducible. These follow from assumptions \ref{cond:coprime_fund_disc}, \ref{cond:ord_red} and \ref{cond:abs_irred} respectively (noting that \cite[Theorem 4]{mazur1978rational} ensures irreducibility (even surjectivity) of the residual representation for semistable curves and $p \geq 11$).

    If $p \nmid d$ and $p$ is good ordinary, then there is nothing further to check.

    If $p \nmid d$ and $p$ is multiplicative, then the conditions in
    Case~\ref{res:multi} follow from assumptions \ref{cond:abs_irred} and
    \ref{cond:ramified_prime} (noting again that \cite[Theorem 4]{mazur1978rational}
    ensures we only need to consider primes $\leq 7$). Here assumption~\ref{cond:ramified_prime} provides a multiplicative prime $q$ of $E$
    at which $\overline{\rho}_{E,p}$ is ramified; since $q \mid N$ and $(D,N)=1$ we have
    $q \nmid D$, so $\chi_d$ is unramified at $q$. Hence $q \parallel N_{E_d}$ and
    $\overline{\rho}_{E_d,p}\vert_{I_q} = \overline{\rho}_{E,p}\vert_{I_q}$ is ramified, so
    $q$ witnesses the ramified prime condition of Case~\ref{res:multi} for $E_d$.

     If $p \nmid d$ and $p$ is supersingular, then the required conditions are established by assumptions \ref{cond:coprime_fund_disc} and \ref{cond:supersingular}.
\end{proof}

\subsection{$2$-part of BSD for $E_d$}\label{ssec:2-part-BSD}

\Cref{thm:main} has a major condition of knowing $\BSD(E_d,2)$. In this section we collect the conditions under which $\BSD(E_d,2)$ can be established, which will feed into \Cref{thm:main} in \Cref{ssec:parts-together}.

The result \cite[Theorem~10.12]{burungale2024zeta} does not explicitly enumerate all of the hypotheses under which one may conclude $\BSD(E_d,2)$. The discussion there refers to \cite[Theorem~1.5]{cai20202} and \cite[Theorem~1]{zhai2016non}; however, in the latter reference the relevant results are presented as Theorems~1.1--1.9. These theorems address closely related situations but involve differing assumptions, notably concerning the sign of the discriminant and the presence or absence of rational \(2\)-torsion. For the reader's convenience, we therefore provide the following result that collects the required hypotheses into a single, self-contained statement.

\begin{theorem}\label{thm: 2-BSD}
    Let $E/\Q$ be an optimal elliptic curve over $\Q$ of conductor $N$ with odd Manin constant and analytic rank $0$. Let $d$ be a squarefree integer coprime to $N$ and congruent to $1$ modulo $4$, and suppose that every prime divisor of $N$ splits in $\Q(\sqrt{d})$. Assume that the $2$-part of BSD holds for $E$.
    
    If one of the following sets of conditions holds, then $L(E_d, 1) \neq 0$, and the $2$-part of BSD holds for $E_d$.

    \begin{enumerate}
    \item If $E(\Q)[2]$ = 0 and $E$ has negative discriminant:
    \begin{enumerate}
        \item for every $p | d$, $p$ is inert in the number field $\Q[x]/(f(x))$, where $f(x)$ is the two-division polynomial of $E$.
        \item $\ord_2(L^{(\mathrm{alg})}(E,1)) = 0$.
    \end{enumerate}
    \item If $E(\Q)[2]$ = 0 and $E$ has positive discriminant:
    \begin{enumerate}
        \item for every $p | d$, $p$ is inert in the number field $\Q[x]/(f(x))$, where $f(x)$ is the two-division polynomial of $E$.
        \item\label{cond:zhai_different_def} $\ord_2(L^{(\mathrm{alg})}(E,1)) = 0$.
        \item $d > 0$.
    \end{enumerate}
    \item If $E(\Q)[2] \cong \Z/2\Z$:
    \begin{enumerate}
        \item for every $p | d$, $p \equiv 1 \Mod{4}$, and $\ord_2(\#\widetilde{E}(\F_p)) = 1$.
        \item $\ord_2(L^{(\mathrm{alg})}(E,1)) = -1$.
        \item $2$ splits in $\Q(\sqrt{d})$.
        \item\label{item:isogenous_curve_sha_torsion} $E'(\Q)[2] \cong \Z/2\Z$ and $\Sha(E')[2] = 0$, where $E' := E/E(\Q)[2]$ is the $2$-isogenous curve of $E$.
    \end{enumerate}
    \end{enumerate}
\end{theorem}

We stress that in this result it is assumed that one knows the $2$-part of BSD for one elliptic curve.

\begin{proof}
    This result is obtained by amalgamating Theorems 1.1 and 1.5 of \cite{cai20202} for the case $E(\Q)[2] \cong \Z/2\Z$; Theorems 1.1 and 1.2 of \cite{zhai2016non} for the case $E(\Q)[2] = 0$ and negative discriminant; and Theorems 1.3 and 1.4 of \cite{zhai2016non} for the case $E(\Q)[2] = 0$ and positive discriminant. There are then some other conditions needed as follows:
    \begin{enumerate}
        \item The requirement that $E$ have rank $0$ is needed due to the use of modular symbols \cite[Chapter 3]{cremona1997algorithms}. This is mentioned in \cite[Section 1]{cai20202} and confirmed by Zhai in private communication.
        \item The condition in \Cref{cond:zhai_different_def} is written as $\ord_2(L^{(\mathrm{alg})}(E,1)) = 1$ in Zhai's paper because he is using a slightly different definition of $L^{(\mathrm{alg})}(E,1)$ than that being used in \cite{cai20202}. In this paper we are using the same definition as in \cite{cai20202}, so we need to subtract the $2$-adic valuation of the number of connected components of $E(\R)$, yielding $0$.
        \item The condition $E'(\Q)[2] \cong \Z/2\Z$ in \Cref{item:isogenous_curve_sha_torsion} is not explicitly written in the theorem statements of \cite{cai20202}, but is mentioned in Remark 1.3 of \emph{loc. cit.} as being a necessary condition for non-emptiness of the set
        \begin{equation}\label{eqn:S}
            \mathcal{S} := \left\{q \equiv 1 \Mod{4} : q \nmid N,\ \text{ord}_2(\#\widetilde{E}(\F_q)) = 1 \,\right\},
        \end{equation}
        which will become the set of primes where the twists $d$ are supported.
        \qedhere
    \end{enumerate}
\end{proof}

\begin{remark}\label{rem:no-full-four-torsion}
    We note that there are no analogues of Theorems 1.2 and 1.4 of \cite{zhai2016non} in the case $E(\Q)[2] \neq 0$. There are such analogues for the case $E(\Q)[2] \cong \Z/2\Z$ in \cite{cai20202}. Therefore, we do not have such results in the full rational $2$-torsion case, and it is for this reason that this case is missing in the above result.
\end{remark}

\subsection{Condition for $\mathcal{S} \neq \emptyset$}\label{ssec:S_non_empty_conds}

As explained in \cite[Remark 1.3]{cai20202}, once one has shown that the set $\mathcal{S}$ in \Cref{eqn:S} is non-empty, then by Chebotarev density arguments, it must have positive density, yielding infinitely many twists in the case $E(\Q)[2] \cong \Z/2\Z$. However, there remains the problem that it may be empty, meaning that there may be no such twists in this case. This section deals with this scenario.

We note that \Cref{item:isogenous_curve_sha_torsion} of \Cref{thm: 2-BSD} is not sufficient (only necessary) to ensure non-emptiness of $\mathcal{S}$. The curve \href{https://www.lmfdb.org/EllipticCurve/Q/85a1/}{\texttt{85a1}} is an example of a curve that passes all of the above checks, yet for all primes $p \equiv 1 \Mod{4}$ up to $10{,}000$, we have $\ord_2(\#\widetilde{E}(\F_p)) \geq 2$ (excluding the bad primes $5$ and $17$ where the valuation is zero). In this case, the existence of an isogenous curve over $\Q(i)$ with full $2$-torsion (as stated on the curve's \href{https://www.lmfdb.org/EllipticCurve/2.0.4.1/7225.5/a/4}{basechange} page) explains why $\ord_2(\#\widetilde{E}(\F_p)) \geq 2$ for all good primes $p \equiv 1 \Mod{4}$.

In general, the main result of \cite{katz1980galois} applied to $E/\Q(i)$ shows that, if $\ord_2(\#\widetilde{E}(\F_p)) \geq 2$ for all good primes $p \equiv 1 \Mod{4}$, then this must be explained by the existence of an isogenous curve $E'/\Q(i)$ that has $4$ dividing $\#E'(\Q(i))_{\mathrm{tors}}$. Whether or not this happens can be expressed in terms of $E/\Q$ itself, as follows.

\begin{proposition}\label{prop:S_non_empty}
    Let $E/\Q$ be an elliptic curve with $E(\Q)[2] \cong \Z/2\Z$. Write \[ E : y^2 = f(x), \]
    and let $x_0 \in \Q$ be such that $(x_0, 0)$ is the rational point of order $2$. 
    
    Then there exists a good prime $p \equiv 1 \Mod{4}$ with $\ord_2(\#\widetilde{E}(\F_p)) = 1$ if and only if none of $f'(x_0)$, $-f'(x_0)$, or $-\Delta$ are squares in $\Q$.
\end{proposition}

Before giving the proof, we recall the link between $f'(x_0)$ and divisibility of the $2$-torsion point by $2$ that arises in the Kummer theory input to the Weak Mordell--Weil theorem. Writing $f(x) = (x - x_0)(x - \alpha)(x - \beta)$ over $\overline{\Q}$, we have $f'(x_0) = (x_0 - \alpha)(x_0 - \beta)$. The halving criterion is that for a field $K$ in which $f$ splits as above with $x_0 \in K$, the point $P_0 = (x_0, 0)$ lies in $2 E(K)$ if and only if $(x_0 - \alpha)(x_0 - \beta) = f'(x_0)$ is a square in $K$. Concretely, if $Q = (x_1, y_1)$ satisfies $2Q = P_0$, then the duplication formula forces $x_1 - x_0$ to be a square root of $f'(x_0)$ (up to a sign convention). We apply this criterion with $K = \F_p$ in condition~(c) below.

\begin{proof}
    ($\Longleftarrow$): Write $f(x) = (x - x_0) g(x)$ where $g$ is a monic quadratic. Since $E(\Q)[2] \cong \Z/2\Z$, the quadratic $g$ has no rational root, so $\mathrm{disc}(g)$ is a non-square in $\Q$. We have $\Delta(E) = 16 f'(x_0)^2 \, \mathrm{disc}(g)$, whence $\Delta$ and $\mathrm{disc}(g)$ agree in $\Q^*/(\Q^*)^2$; in particular, $\Delta$ is a non-square.

    We seek a prime $p$ satisfying:
    \begin{enumerate}
        \item[(a)] $p \equiv 1 \pmod 4$, i.e., $-1$ is a square mod $p$;
        \item[(b)] $\Delta$ is a non-square mod $p$ (so $g$ is irreducible mod $p$, making $(x_0, 0)$ the only $2$-torsion point of $\widetilde{E}(\F_p)$, and hence the $2$-Sylow cyclic);
        \item[(c)] $f'(x_0)$ is a non-square mod $p$, so that $(x_0, 0) \notin 2\widetilde{E}(\F_p)$ by the halving criterion applied over $\F_p$ as in the preamble (note that $f'(x_0) \in \F_p$ even though the roots of $g$ are only in $\F_{p^2}$); hence the $2$-Sylow has order exactly $2$.
    \end{enumerate}
    Equivalently, we want a prime that splits in $\Q(i)$ and is inert in both $\Q(\sqrt{f'(x_0)})$ and $\Q(\sqrt{\Delta})$.

    Let $L = \Q(i, \sqrt{f'(x_0)}, \sqrt{\Delta})$ and $G = \mathrm{Gal}(L/\Q)$. The action on $(i, \sqrt{f'(x_0)}, \sqrt{\Delta})$ embeds $G \hookrightarrow \{\pm 1\}^3$. It suffices to show that $\sigma = (+1, -1, -1) \in G$: Chebotarev then produces infinitely many primes with $\mathrm{Frob}_p = \sigma$, each satisfying (a), (b), (c).

    The image of $G$ in $\{\pm 1\}^3$ is cut out by the multiplicative relations among $\{-1, f'(x_0), \Delta\}$ modulo squares in $\Q^*$. There are seven nontrivial products to consider:
    \begin{itemize}
        \item $-1$, $f'(x_0)$, $\Delta$ a square: ruled out (trivially, by hypothesis, and by the discriminant computation above, respectively);
        \item $-f'(x_0)$, $-\Delta$ a square: ruled out by hypothesis;
        \item $f'(x_0) \cdot \Delta$ or $-f'(x_0) \cdot \Delta$ a square: not ruled out by the hypothesis.
    \end{itemize}

    We check that the two relations not excluded by hypothesis are compatible with~$\sigma$.

    If $f'(x_0) \cdot \Delta$ is a square in $\Q$, then $\sqrt{\Delta} \in \Q(\sqrt{f'(x_0)})$, forcing every element of $G$ to act with the same sign on $\sqrt{f'(x_0)}$ and $\sqrt{\Delta}$. Our $\sigma$ acts as $(-1, -1)$ on this pair, so the relation is satisfied.

    If $-f'(x_0) \cdot \Delta$ is a square, then $\sqrt{\Delta} \in \Q(i, \sqrt{f'(x_0)})$ with $\sqrt{\Delta} = i \cdot \alpha$ for some $\alpha \in \Q(\sqrt{f'(x_0)})$, so every element of $G$ acts on $\sqrt{\Delta}$ by the product of its actions on $i$ and $\sqrt{f'(x_0)}$. For $\sigma$, this product is $(+1)(-1) = -1$, matching its action on $\sqrt{\Delta}$.

    Thus $\sigma \in G$, and Chebotarev density yields infinitely many good primes $p$ with $\mathrm{Frob}_p = \sigma$. For each such $p$, conditions (a)--(c) hold, so $\ord_2(\#\widetilde{E}(\F_p)) = 1$.
    
    ($\Longrightarrow$): If there exists such a prime $p$, it must satisfy $p \equiv 1 \pmod 4$ and be inert in both $\mathbb{Q}(\sqrt{f'(x_0)})$ and $\mathbb{Q}(\sqrt{\Delta})$. This means the Frobenius element $(+1, -1, -1)$ actually exists in the Galois group. If any of $f'(x_0)$, $-f'(x_0)$, or $-\Delta$ were squares in $\mathbb{Q}$, that specific Frobenius element would be algebraically impossible (it would force a contradiction in the signs). Thus, none of them can be squares.
\end{proof}

\begin{remark}
We have verified Proposition~\ref{prop:S_non_empty} computationally on every elliptic curve~$E$ that reaches this stage of our algorithm. Whenever $E$ satisfies the algebraic criterion of the proposition, we exhibit an explicit prime $q \leq 10^4$ with $q \equiv 1 \pmod{4}$, $q \nmid N$, and $\mathrm{ord}_2\bigl(\#E(\mathbb{F}_q)\bigr) = 1$, thereby witnessing $\mathcal{S} \neq \emptyset$. Conversely, whenever $E$ fails the criterion, we have checked that no such prime exists in the range $q \leq 10^4$, verifying the conclusion $\mathcal{S} = \emptyset$.
\end{remark}

The criterion in the above result will be used in the algorithms to follow to ensure non-emptiness of $\mathcal{S}$.

\subsection{Putting the parts together}\label{ssec:parts-together}
We now seek to combine \Cref{thm:main} and \Cref{thm: 2-BSD}. Specifically, \Cref{thm: 2-BSD} will give condition \ref{list-item:bsd_at_2_for_twist} at the expense of additional assumptions on $E$ and $d$. Moreover, we use \Cref{prop:S_non_empty} to strengthen \Cref{item:isogenous_curve_sha_torsion} of \Cref{thm: 2-BSD} to an actual method for ensuring infinitude of $d$ in that case. To keep track of the many such assumptions, we provide the following result that allows for an algorithmic implementation and justifies the title of this paper (specifically, the ``infinitely many''). It separates the conditions on $E$ from those on $d$, reflecting what is implicitly happening in \cite[Theorem 1.7]{burungale2024zeta}.

\begin{theorem}\label{thm:main_alg}
    Let $E/\Q$ be an elliptic curve of conductor $N$ that satisfies the following conditions:
    \begin{enumerate}
        \item $E$ is semistable.
        \item $a_3(E) \in \left\{-2, -1, 0, 1, 2\right\}$.
        \item $E$ does not admit a rational $p$-isogeny for $p \in \left\{3,5,7\right\}$.
        \item\label{main-thm-ramified-prime}  For any prime $p | N$, there exists a multiplicative prime $q \neq p$ at which $E[p]$ is ramified.
        \item $E$ is optimal in its isogeny class.
        \item $E$ has analytic rank 0.
        \item $\BSD(E, 2)$ is true.
        \item One of the following two sets of conditions hold.
        \begin{enumerate}
            \item\label{tors1} $E(\Q)[2] = 0$, $\ord_2(L^{(\mathrm{alg})}(E,1)) = 0$
            \item\label{tors3} $E(\Q)[2] \cong \Z/2\Z$; $\ord_2(L^{(\mathrm{alg})}(E,1)) = -1$; writing $E : y^2 = f(x)$ and the $2$-torsion point as $(x_0,0)$, none of $f'(x_0)$, $-f'(x_0)$ or $-\Delta_E$ are squares in $\Q$; $\Sha(E')[2] = 0$, where $E' := E/E(\Q)[2]$ is the $2$-isogenous curve of $E$.
        \end{enumerate}
    \end{enumerate}
    Then there exist infinitely many squarefree integers $d$ satisfying the following:
    \begin{enumerate}
        \item $(d, 3N) = 1$.
        \item $d \equiv 1 \Mod{4}$.
        \item $E$ has ordinary reduction at prime divisors of $d$.
        \item 
        \begin{enumerate}
            \item in Item~\ref{tors1}, for every $p | d$, $p$ is inert in the number field $\Q[x]/(f(x))$, where $f(x)$ is the two-division polynomial of $E$; for every $p \mid N$, $p$ splits in $\Q(\sqrt{d})$; if $\Delta_E > 0$, then $d > 0$.
            \item in Item~\ref{tors3}, for every $p | d$, $p \equiv 1 \Mod{4}$, and $\ord_2(\#\widetilde{E}(\F_p)) = 1$; and for every $p \mid 2N$, $p$ splits in $\Q(\sqrt{d})$.
        \end{enumerate}
    \end{enumerate}
    Therefore, by \Cref{thm:main} and \Cref{thm: 2-BSD}, such $E$ as above admit infinitely many quadratic twists $E_d$ that unconditionally satisfy BSD.
\end{theorem}

Note that we have dropped the condition that $E$ has odd Manin constant, since the main theorem of \cite{vcesnavivcius2018manin} proves that it is $1$ for optimal semistable curves. Also note that the condition $d \equiv 1 \Mod{4}$ implies that $D = d$, so the condition $(d, 3N) = 1$ is equivalent to \Cref{cond:coprime_fund_disc}. The infinitude of $d$ is stated in \cite[Theorem 1.7]{burungale2024zeta} as relying on the two papers \cite{cai20202,zhai2016non}. These two papers establish infinitude of a family of twists by using the Chebotarev density theorem together with their modular symbols arguments. We provide some more details on this as follows.

\begin{proof}
Define
\[
    \mathcal{S}_0 := \{\ell \text{ prime} :
        \ell \nmid 6N,\ E \text{ ordinary at } \ell\}.
\]
Since $E$ is semistable over $\Q$, it is not CM; a theorem of Serre then gives that the
supersingular set of $E$ has density $0$. Together with the finite exclusions, $\mathcal{S}_0$ has natural density $1$. For any squarefree $d > 0$ whose prime factors lie in $\mathcal{S}_0$, conclusions (1) and (3) of the theorem therefore hold automatically.

We construct the family separately in cases~\ref{tors1} and~\ref{tors3}; in each, $d$ will be a product of two primes drawn from an appropriate Chebotarev set derived from $\mathcal{S}_0$.

\medskip
\noindent\textbf{Case~\ref{tors1}.}
Since $E(\Q)[2] = 0$, the $2$-division polynomial $f$ is irreducible
of degree $3$. Let $F = \Q[x]/(f(x))$ and $K_2 = \Q(E[2])$ its Galois closure. Then $\Gal(K_2/\Q) \in \{S_3, A_3\}$, and the condition ``$\ell$ inert in
$F$'' is equivalent to ``$\Frob_\ell$ has order~$3$ in
$\Gal(K_2/\Q)$''. Order-$3$ elements form a non-empty union of
conjugacy classes in either of those two possibilities ($S_3$ and $A_3$), so by Chebotarev applied to
$K_2/\Q$,
\[
    \mathcal{S} := \{\ell \in \mathcal{S}_0 : \ell \text{ inert in } F\}
\]
has positive density. (This is the Chebotarev set used in
\cite[Theorems~1.1, 1.3]{zhai2016non}.)

The splitting behaviour of an odd prime $p \mid N$ in $\Q(\sqrt{d})$
is controlled by the Legendre symbol $\bigl(\tfrac{d}{p}\bigr)$, and
the splitting behaviour of $2$ (relevant when $2 \mid N$) is controlled
by $d \bmod 8$. For each $\ell \in \mathcal{S}$ we therefore record
this data as the \emph{signature}
\[
    \sigma(\ell) := \Bigl(\ell \bmod 8,\
    \bigl(\tfrac{\ell}{p}\bigr)_{p \mid N,\ p \text{ odd}}\Bigr) \in
    (\Z/8\Z)^\times \times \{\pm 1\}^{\omega_o(N)},
\]
where $\omega_o(N)$ is the number of odd prime divisors of $N$. The
codomain is finite, while $\mathcal{S}$ is infinite, so some value
$\sigma_0$ is attained by infinitely many $\ell \in \mathcal{S}$;
write $\mathcal{S}_{\sigma_0}$ for this infinite subset. Choose
distinct $\ell_1, \ell_2 \in \mathcal{S}_{\sigma_0}$ and set
$d := \ell_1 \ell_2$. We verify the conclusions of the theorem in
turn.

Both $\ell_1$ and $\ell_2$ lie in $\mathcal{S}$, so both are inert in $F$. Since $\ell_1 \equiv \ell_2 \pmod 8$, we have $d \equiv \ell_1^2
\pmod 8$, and every odd square is $\equiv 1 \pmod 8$; hence $d \equiv
1 \pmod 8$. This gives conclusion (2), and when $2 \mid N$ it also
gives that $2$ splits in $\Q(\sqrt{d})$. For each odd $p \mid N$, the
Legendre symbol of the product is the product of the symbols, which
are equal by construction, so
$\bigl(\tfrac{d}{p}\bigr) = \bigl(\tfrac{\ell_i}{p}\bigr)^2 = 1$ and
$p$ splits in $\Q(\sqrt{d})$. Combined with the previous sentence,
this gives the second half of conclusion (5a). Finally,
$d = \ell_1 \ell_2 > 0$, so the positivity clause of (5a) is
automatic.

Letting $(\ell_1, \ell_2)$ range over distinct pairs in
$\mathcal{S}_{\sigma_0}$ produces infinitely many such $d$.

\medskip
\noindent\textbf{Case~\ref{tors3}.}
Since $E(\Q)[2] \cong \Z/2\Z$, the $2$-division polynomial factors as
$f(x) = (x - x_0)g(x)$ with $g$ an irreducible quadratic. The field
$\Q(E[2])$ is the splitting field of $g$, and the discriminant identity
$\Delta_E = 16 f'(x_0)^2 \cdot \mathrm{disc}(g)$ shows that this is
$\Q(\sqrt{\Delta_E})$. A similar direct computation with the explicit
$2$-isogeny gives $\Q(E'[2]) = \Q(\sqrt{f'(x_0)})$, where
$E' := E/\langle (x_0, 0) \rangle$ is the $2$-isogenous curve.

By \cite[Remark~1.3]{cai20202}, invoking \cite[Lemma~4.1]{kriz2019prime},
for a prime $\ell \nmid 2N$ with $\ell \equiv 1 \pmod 4$, the condition
$\ord_2(\#\widetilde{E}(\F_\ell)) = 1$ is equivalent to $\ell$ being
inert in \emph{both} $\Q(E[2])$ and $\Q(E'[2])$. Set
\[
    \mathcal{S}' := \{\ell \in \mathcal{S}_0 :
        \ell \equiv 1 \Mod 4,\
        \ell \text{ inert in } \Q(E[2]),\
        \ell \text{ inert in } \Q(E'[2])\}.
\]
The hypothesis that $-\Delta_E$, $f'(x_0)$, and $-f'(x_0)$ are not
rational squares ensure that $\Q(\sqrt{\Delta_E})$ and
$\Q(\sqrt{f'(x_0)})$ are both nontrivial quadratic extensions of $\Q$,
and that both are distinct from $\Q(i)$. The fields $\Q(\sqrt{\Delta_E})$
and $\Q(\sqrt{f'(x_0)})$ may or may not coincide --- they coincide
precisely when $\Delta_E \cdot f'(x_0)$ is a rational square --- but in
either situation the conditions defining $\mathcal{S}'$ specify a single
Frobenius element in $\Gal(\Q(i, \sqrt{\Delta_E}, \sqrt{f'(x_0)})/\Q)$.
By Chebotarev, $\mathcal{S}'$ has positive density. This is the set $S$
of \cite[Theorem~1.1]{cai20202}.

As in case~\ref{tors1}, the splitting of $2$ and of the odd primes
$p \mid N$ in $\Q(\sqrt{d})$ is controlled by $d \bmod 8$ and by the
Legendre symbols $\bigl(\tfrac{d}{p}\bigr)$, respectively. For each
$\ell \in \mathcal{S}'$, define as before the signature
\[
    \sigma(\ell) := \Bigl(\ell \bmod 8,\
    \bigl(\tfrac{\ell}{p}\bigr)_{p \mid N,\ p \text{ odd}}\Bigr) \in
    \{1, 5\} \times \{\pm 1\}^{\omega_o(N)},
\]
where the first coordinate is restricted to $\{1, 5\}$ because
$\ell \equiv 1 \Mod 4$. The codomain is finite while $\mathcal{S}'$ is
infinite, so some value $\sigma_0$ is attained by infinitely many
$\ell \in \mathcal{S}'$; write $\mathcal{S}'_{\sigma_0}$ for this
infinite subset. Choose distinct
$\ell_1, \ell_2 \in \mathcal{S}'_{\sigma_0}$ and set
$d := \ell_1 \ell_2$. We verify the conclusions of the theorem in turn.

Both $\ell_1$ and $\ell_2$ lie in $\mathcal{S}'$, so each satisfies
$\ell_i \equiv 1 \Mod 4$ and $\ord_2(\#\widetilde{E}(\F_{\ell_i})) = 1$;
this gives the first part of conclusion (5b). Since
$\ell_1 \equiv \ell_2 \pmod 8$, we have $d \equiv \ell_1^2 \pmod 8$, and
both possible values of $\ell_1 \bmod 8$ square to $1$ (namely
$1^2 = 1$ and $5^2 = 25 \equiv 1$); hence $d \equiv 1 \pmod 8$. This
gives conclusion (2) and also that $2$ splits in $\Q(\sqrt{d})$. For
each odd $p \mid N$, the Legendre symbol of the product is the product
of the symbols, which are equal by construction, so
$\bigl(\tfrac{d}{p}\bigr) = \bigl(\tfrac{\ell_i}{p}\bigr)^2 = 1$ and $p$
splits in $\Q(\sqrt{d})$. Combined with the splitting at $2$, this gives
the second part of conclusion (5b).

Letting $(\ell_1, \ell_2)$ range over distinct pairs in
$\mathcal{S}'_{\sigma_0}$ produces infinitely many $d$.
\end{proof}

\subsection{Algorithms}\label{ssec:algs}

\Cref{thm:main_alg} now readily translates into two algorithms: \Cref{algo: base_curve}, implemented in \texttt{infinite\_bsd/Algorithm1.py}, that checks the conditions on $E$; and \Cref{algo: twists}, implemented in \texttt{infinite\_bsd/Algorithm2.py}, that takes an $E$ successfully checked by \Cref{algo: base_curve}, as well as a squarefree integer $d$, and \texttt{True}/\texttt{Unknown} depending on whether or not the pair $(E,d)$ satisfies all the conditions listed in \Cref{thm:main_alg}.

\begin{algorithm}
\caption{Filter Eligible Elliptic Curves}
\begin{algorithmic}[1]{\label{algo: base_curve}}
\Require Elliptic curve $E$
\State\label{alg-step:squarefree}\textbf{Check:} $N$ is squarefree and not prime
\State \textbf{Check:} $a_3(E) \in \left\{-2, -1, 0, 1, 2\right\}$
\State\label{alg-step:no_p_isogeny}\textbf{Check:} $\forall p \in \{3, 5, 7\}$, $E$ does not admit a rational $p$-isogeny
\State\label{alg-step:discriminant}\textbf{Check:} $\forall p \mid N, \exists \ q \mid N$, $q \ne p$, s.t. $p \nmid \text{ord}_q(\Delta_E)$
\State \textbf{Check:} $E$ is optimal in its isogeny class
\State \textbf{Check:} \text{rank}($E$) = 0
\If{$E(\mathbb{Q})[2] \cong \mathbb{Z}/2\mathbb{Z}$}
    \State \textbf{Check:} $\text{ord}_2(L^{(\mathrm{alg})}(E,1)) = -1$
    \State Let $E' := E / E(\mathbb{Q})[2]$
    \State \textbf{Check:} $\Sha(E')[2] = 0$ and $E'(\Q)[2] \cong \Z/2\Z$
    \State Write $E : y^2 = f(x)$ and let $(x_0,0) \in E(\mathbb{Q})[2]$
    \State \textbf{Check:} none of $f'(x_0)$, $-f'(x_0)$, or $-\Delta_E$ is a square in $\mathbb{Q}$
\ElsIf{$E(\mathbb{Q})[2] = 0$}
    \State \textbf{Check:} $\text{ord}_2(L^{(\mathrm{alg})}(E,1))) = 0$
\EndIf
\If{2-part of BSD holds}\label{alg-step:2-part-bsd} 
    \State \textbf{Accept} $E$
\EndIf
\end{algorithmic}
\end{algorithm}

\begin{remark}\label{rem:alg1}
    \begin{enumerate}
        \item The requirement that $N$ not be prime in Step~\ref{alg-step:squarefree} is forced by the ramified prime condition~\ref{main-thm-ramified-prime} in \Cref{thm:main_alg}. By using the theory of Tate curves, this condition is seen to be equivalent to Step~\ref{alg-step:discriminant}. We are not claiming that $N$ being composite ensures the ramified prime condition holds.
        \item The check on isogeny degrees and optimality are database checks, since this information is stored in the LMFDB for each curve. We refer to \cite{cremona1997algorithms} for how these are computed.
        \item Step~\ref{alg-step:2-part-bsd} certainly warrants some discussion. The suite of papers mentioned in the Introduction (those that established full BSD for all curves of analytic rank $0$ or $1$ up to conductor $5{,}000$) furnished a very useful Sage function \texttt{prove\_BSD} that attempts to unconditionally prove BSD for a given curve, and returns the primes where verification of $\BSD(E,p)$ fails. We extract the source code for this specifically for $p = 2$ into a helper function \texttt{check\_BSD\_at\_2} and run this at the end of the routine to minimize the number of curves we use this expensive function on. This function returns True if and only if $\BSD(E,2)$ can be unconditionally shown to be true via $2$-descent. $2$-descent alone settles $\BSD(E,2)$ only when $\ord_2(\#\Sha_{\mathrm{an}}(E)) = 0$. Fortunately for us, this is the case for all curves in the LMFDB that reach this Step in the above algorithm. (In other words, $\#\Sha_{\mathrm{an}}(E)$ is always odd for us.)  For positive values of $\ord_2(\#\Sha_{\mathrm{an}}(E))$, one would need to carry out higher $2$-power descent.
    \end{enumerate}
\end{remark}

For an elliptic curve $E$ that is in the output of \Cref{algo: base_curve}, we next wish to identify the twists in a given range $[-B, B]$ that satisfy BSD. The conditions on $d$ here listed in the second half of \Cref{thm:main_alg} can again be extracted into a procedural \Cref{algo: twists} as follows.

\begin{algorithm}
\caption{Determine if $E_d$ satisfies BSD}
\begin{algorithmic}[1]{\label{algo: twists}}
\Require Elliptic curve $E$ output by \Cref{algo: base_curve}
\Require Squarefree integer $d$
\State $N \gets$ conductor($E$)
\State \textbf{Check:} $(d,3N) = 1$
\State \textbf{Check:} $\forall p | d, \ p \nmid a_p(E)$
\If{$E(\mathbb{Q})[2] \cong \Z/2\Z$}
    \State \textbf{Check:} $\forall p | d, \ p \equiv 1 \Mod{4}$ and $\ord_2(\#\widetilde{E}(\F_p)) = 1$
    \State\label{alg-step:quadratic-ram-1} \textbf{Check:} $d \equiv 1\pmod{8}$ and if $p \mid N$ is odd, then $\Big(\frac{\Delta_{\Q\sqrt{d}}}{p} \Big) = 1$
    \State \Return True
\ElsIf{$E(\mathbb{Q})[2] = 0$}
    \State \textbf{Check:} $d \equiv 1\pmod{4}$
    \State $f(x) \gets$ $2$-division polynomial of $E$
    \State $F \gets \Q[x]/(f(x))$
    \State \textbf{Check:} $\forall p | d$, $p$ is inert in $F$
    \State\label{alg-step:quadratic-ram-2} \textbf{Check:} $\forall p | N$, $\Big(\frac{\Delta_{\Q\sqrt{d}}}{p} \Big) = 1$
    \If{$\Delta_E > 0$}
        \State \textbf{Check:} $d > 0$
        \State \Return True
    \Else
        \State \Return True
    \EndIf
\EndIf
\State \Return Unknown
\end{algorithmic}
\end{algorithm}

\begin{remark}
For a fixed curve \(E\), we do not claim to characterize all quadratic twists \(E_d\) satisfying BSD. What our results give is an explicit, effectively enumerable subfamily of such \(d\)s. Equivalently, one may generate admissible \(d\)s by choosing prime factors from the relevant Chebotarev sets and imposing the required congruence and splitting conditions using \Cref{algo: twists}. Thus the procedure is more structured than testing arbitrary squarefree \(d\)'s one by one, although in bounded computations we still enumerate candidates within this explicitly described admissible set.
\end{remark}

We executed these algorithms on all elliptic curves in the LMFDB database up to conductor $500{,}000$; the output of doing this can be found at \path{infinite_bsd/output/twists_of_ec_labels_500k.txt} \cite{banwait2024quadtwist}. This lists, for each elliptic curve that passed \Cref{algo: base_curve}, the BSD-satisfying twists in the range $[-1000,1000]$. Readers wishing to see only the elliptic curves and not the twists may refer to \texttt{infinite\_bsd/output/ec\_labels\_500k.txt}.

Of the $3{,}064{,}705$ elliptic curves in the LMFDB with conductor below $500{,}000$, exactly $1{,}170{,}876$ $(\approx 38.2\%)$ have analytic rank 0. Of those, $274{,}888$ ($\approx 23.5\%$ of the analytic-rank-0 curves, $\approx 9.0\%$ of all curves with $N < 500{,}000$) are semistable, and after restricting further to those whose conductor has at least two prime divisors and that are optimal representatives of their isogeny classes, $178{,}364$ curves remain. Algorithm 1 accepts $36{,}687$ of these — approximately $20.5\%$ of the optimal semistable analytic-rank-0 candidates, and $1.20\%$ of all elliptic curves in the LMFDB with $N < 500,000$.

\subsection{The curves up to conductor 150}\label{ssec:bsd_curves_cond_150}

We list here the elliptic curve Cremona labels of the curves up to conductor 150 that our algorithm shows to admit infinitely many twists that unconditionally satisfy strong BSD. We present this in the same fashion as \cite[Example 1]{burungale2024zeta}, namely, two bullet points, the first for curves that use the results of \cite{cai20202} for establishing the $2$-part of BSD for the twist, and the second for curves that use the results in \cite{zhai2016non}:
\begin{itemize}
    \item 46a1, 69a1, 77c1, 94a1, 114b1, 141b1, 142c1;
    \item 106d1, 115a1, 118c1, 118d1, 141e1.
\end{itemize}

We compare these lists with the curves in \cite[Example 1]{burungale2024zeta}.

\begin{enumerate}
    \item The curves in our respective first bullet points match entirely.
    \item The curves in our second bullet form a strict subset of the curves in their second bullet point. The curves there that we are not listing here are those with Cremona label
    62a1, 66b1, 105a1, and 141c1. To investigate the discrepancy, we implemented the script \path{diagnose_curve_discrepancy.py} that produces \path{output/discrepancy_report.txt} showing exactly which of the conditions are being violated. We find that all four fail for the same reason: the condition for non-emptiness of $\mathcal{S}$ in \Cref{prop:S_non_empty} is violated; i.e. the set $\mathcal{S}$ is empty. (This is also confirmed by searching for primes up to $10{,}000$.) Since $\mathcal{S}$ is the set of prime support for the infinitely many twists, these four curves have not been shown to admit infinitely many twists satisfying BSD. 
    
    Moreover, all four of these curves have $E(\Q)[2] \cong \Z/2\Z$. For this second bullet point, the authors of that paper are relying on the theorems in \cite{zhai2016non} to establish the existence of infinitely many twists that unconditionally satisfy the $2$-part of BSD; and while there are such theorems in the $E(\Q)[2] = 0$ case, there are no such theorems in \cite{zhai2016non} that establish such a result in the case $E(\Q)[2] \neq 0$, as also mentioned in \Cref{rem:no-full-four-torsion}.\footnote{
    The closest statement is Theorem~1.6 of \emph{loc.\ cit.}, which applies only to Neumann--Setzer curves; however, this case is excluded in our setting, since such curves necessarily have prime conductor, which is precluded as in \Cref{rem:alg1}.
    }
\end{enumerate}

\section{Verifying the conjecture of Radziwiłł and Soundararajan}\label{sec:sound_conjecture}

Having established an algorithm to identify families of quadratic twists where the full Birch and Swinnerton-Dyer (BSD) conjecture holds, we now consider the statistical behavior of the Shafarevich--Tate group $\Sha(E_d)$ within these families. In particular, we investigate whether the distribution of the order of $\Sha(E_d)$ conforms to the probabilistic predictions formulated by Radziwiłł and Soundararajan.

\subsection{Background on the Conjecture}
In \cite{radziwill2015moments}, Radziwiłł and Soundararajan investigated the distribution of central $L$-values of quadratic twists of elliptic curves. Based on results from random matrix theory (specifically the Keating--Snaith conjecture), they formulated a conjecture describing the distribution of the size of $\Sha(E_d)$ for rank zero twists, that we now briefly recall.

Let $E$ be an elliptic curve over $\mathbb{Q}$ given by the model $y^2 = f(x)$ for a monic squarefree cubic polynomial $f \in \mathbb{Z}[x]$. Let $K$ be the splitting field of $f$ over $\mathbb{Q}$ and $G = \text{Gal}(K/\mathbb{Q})$ viewed as a subgroup of $S_3$. For $g \in G$, let $c(g)$ denote $1 + |\text{Fix}(g)|$, where $\text{Fix}(g)$ is the number of fixed points of $g$ acting on the roots of $f$. We define the arithmetic mean $\mu(E)$ and variance $\sigma(E)^2$ associated to the Galois action as:
\[
\mu(E) = -\frac{1}{2} - \frac{1}{|G|} \sum_{g \in G} \log c(g) \quad \text{and} \quad \sigma(E)^2 = 1 + \frac{1}{|G|} \sum_{g \in G} (\log c(g))^2.
\]
The Radziwiłł--Soundararajan conjecture predicts that as $d$ ranges over those fundamental discriminants where the twist $E_d$ has root number +1 (and hence even analytic rank), the distribution of the normalized values
\[
Z_d = \frac{\log \left( |\Sha(E_d)| / \sqrt{|d|} \right) - \mu(E) \log \log |d|}{\sqrt{\sigma(E)^2 \log \log |d|}}
\]
converges to a standard Gaussian distribution $\mathcal{N}(0,1)$ as $|d| \to \infty$. 

The above is sufficient for our purposes here, although we note that there is a more precise version given by the authors as Conjecture 1 in their paper.

\subsection{Verifying the conjecture}\label{subsec:unconditional_family}

We first wish to verify the Radziwill--Soundararajan conjecture numerically for all twists with root number $+1$, assuming BSD. We focus on the elliptic curve $E=\texttt{46a1}$ given by $y^2=x^3-163x-930$ that arises in \Cref{ssec:bsd_curves_cond_150}; the other curves listed there exhibit the same pattern. For this curve, we compute
\[
Z_d^{\textup{an}} := \frac{\log\!\bigl(|\Sha(E_d)^{\mathrm{an}}|/\sqrt{|d|}\bigr)-\mu(E)\log\log|d|}
{\sqrt{\sigma(E)^2\log\log|d|}},
\]
where as usual
\[
|\Sha(E_d)^{\mathrm{an}}| := \frac{L^{(r)}(E,1)/r! \cdot |E(\Q)_{\mathrm{tors}}|^2}{\Omega(E)\cdot\prod_pc_p\cdot \Reg(E)}
\]
is the order of $\Sha(E_d)$ derived from the strong BSD formula.

Figure \ref{fig:baseline_distribution} displays the results for $X = 100{,}000$, where $X$ is the bound on $d$ (excluding 472 values where the computation for the order of $\Sha(E)^{\textup{an}}$ took longer than 10 seconds and so was abandoned).

\begin{figure}[h!] 
\centering 
\includegraphics[width=0.8\textwidth]{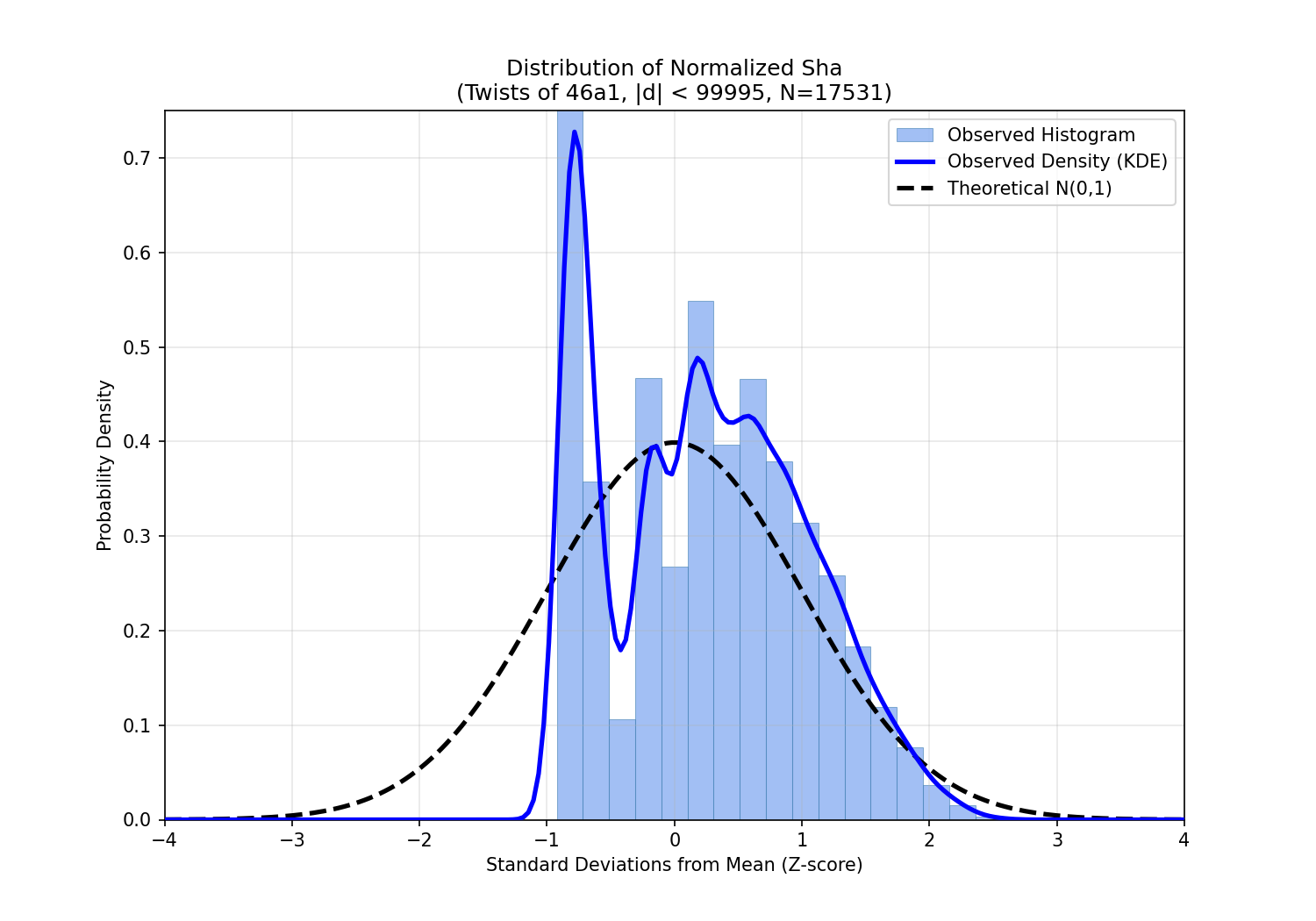} \caption{Distribution of $17{,}531$ $Z_d^{\textup{an}}$ values for generic twists of \texttt{46a1} with $|d| \le 100{,}000$. The histogram (light blue) shows the raw counts, while the Kernel Density Estimate (solid blue line) provides a smooth approximation that closely tracks the theoretical standard Gaussian prediction (black dashed). A time-lapse visualization of the convergence to $\mathcal{N}(0,1)$ is available at \href{https://radzi-sound-conjecture.netlify.app/46a1_maxd100000_nf100_20260116_185105_pdf}{this link}.} \label{fig:baseline_distribution}
\end{figure}

\subsubsection{Statistical Diagnostics for Normality}\label{subsubsec:normal_test}

To quantify the convergence of the empirical distributions of $Z_d$ to $\mathcal{N}(0,1)$, we compute two standard metrics: the Kolmogorov--Smirnov (K--S) distance (cf.~\cite{massey1951kolmogorov}) and the Wasserstein distance (cf.~\cite{villani2008optimal}). The KS distance measures the maximum vertical deviation between CDFs, whereas the Wasserstein distance quantifies the overall discrepancy between distributions by computing the minimal cost of transporting mass via optimal transport. Both are widely used measures of distributional discrepancy (cf.~\cite{devroye2013probabilistic, gibbs2002choosing}). At $X=10{,}000$, both distances exhibit noticeable deviations from the standard normal distribution, consistent with finite-sample effects. However, as $X$ increases, the K--S and Wasserstein distances decrease systematically across all examples, providing evidence of convergence toward the Gaussian limit predicted by Radziwiłł--Soundararajan.

Figure~\ref{fig:46a1_diagnostics} illustrates this behavior for \texttt{46a1} as a representative case, where both distance metrics decrease as the discriminant bound grows to $X=100{,}000$.

\begin{figure}[h!]
\centering

\begin{subfigure}[t]{0.48\textwidth}
    \centering
    \includegraphics[width=\textwidth]{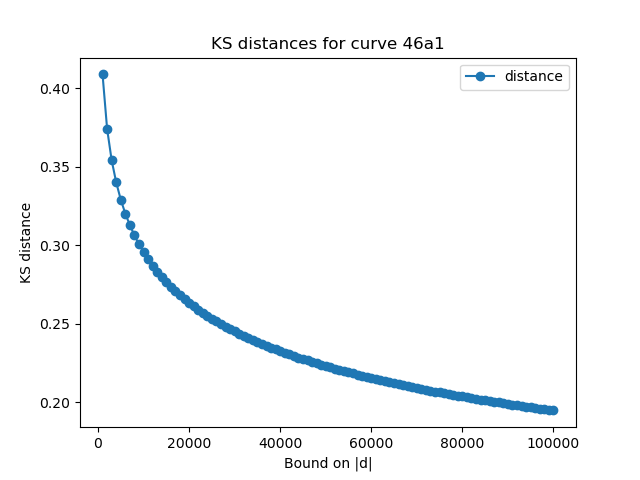}
    \caption{KS distance}
\end{subfigure}
\hfill
\begin{subfigure}[t]{0.48\textwidth}
    \centering
    \includegraphics[width=\textwidth]{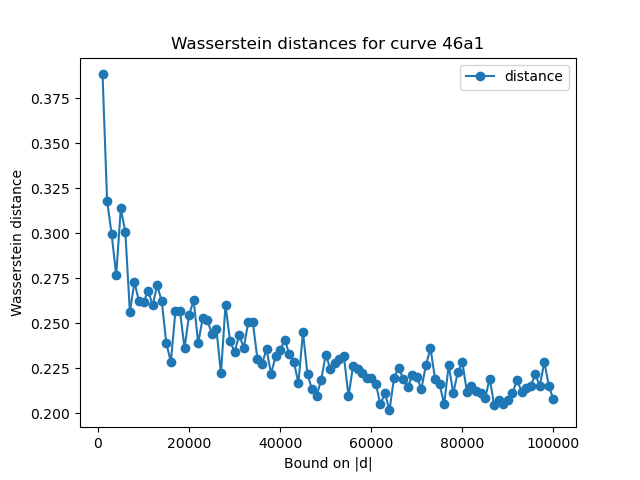}
    \caption{Wasserstein distance}
\end{subfigure}{\label{fig:KS}}

\caption{Goodness-of-fit diagnostics for curve 46a1 as functions of the truncation bound on $|d|$ where the maximum bound $X$ increases to $100,000$. The KS and Wasserstein distances decrease with increasing bound indicating improved distributional agreement to normality.}
\label{fig:46a1_diagnostics}
\end{figure}

\subsection{Observed Deviation in the BSD Family}\label{subsec:deviation}

Having validated the generic case, we now turn to our primary object of interest: the family of twists identified in Section \ref{sec : p-BSD}. These are the twists for which the full BSD conjecture is unconditionally proven. We refer to this set of twists by $\mathcal{D}_{\mathrm{BSD}}$. We note that this set is defined by explicitly computable local conditions, implemented in the \texttt{get\_admissible\_twists\_CLZ} and \texttt{get\_admissible\_twists\_Zhai} functions in \path{ants_xvii/infinite_bsd/Algorithm2.py}.

A natural question is whether the distribution of $|\Sha(E_d)|$ for $d\in\mathcal{D}_{\mathrm{BSD}}$ follows the generic Radziwiłł--Soundararajan (RS) prediction. We take one example: for the elliptic curve \texttt{46a1}, we generated the set $\mathcal{D}_{\mathrm{BSD}} \cap [-X, X]$ with $X = 300,000$ using the criteria when $E(\mathbb{Q})[2] \cong \Z/2\Z$ from \cite{cai20202}, implemented in our \Cref{algo: twists}. This yielded a subset of $N=1008$ fundamental discriminants. Recall that unlike the generic family, membership in $\mathcal{D}_{\mathrm{BSD}}$ requires satisfying stringent arithmetic constraints (see \Cref{algo: twists}). Specifically, for every $d \in \mathcal{D}_{\mathrm{BSD}}$, we require:
\begin{itemize}
    \item $d$ is squarefree and $d \equiv 1 \pmod 8$;
    \item Every prime factor $p \mid d$ satisfies $p \equiv 1 \pmod 4$;
    \item For all odd primes $q$ dividing the conductor $N=46$, the Legendre symbol satisfies $\left(\frac{d}{q}\right) = 1$.
\end{itemize}

In particular, these conditions imply that $d$ must be positive: if $d$ were negative, the condition on its prime factors would imply $d \equiv 3 \pmod 4$, contradicting the requirement that $d \equiv 1 \pmod 8$. Consequently, our family $\mathcal{D}_{\mathrm{BSD}}$ consists entirely of real quadratic twists, in contrast to the generic case which includes both real and imaginary fields.

As in \cite[Section 8.1]{dabrowski2016behaviour} we computed $|\Sha(E_d)|$ from the BSD formula (which we know unconditionally for these values of $d$) for rank 0 curves (which we know to be the case from Algorithms 1 and 2 from \Cref{sec : p-BSD}) as follows: \[ |\Sha(E_d)| = \frac{L(E,1)\cdot|E_d(\Q)_{\mathrm{tors}}|^2}{\Omega(E)\cdot\prod_pc_p}.\]
We then computed $Z_d$; the resulting distribution is shown in Figure \ref{fig:biased_distribution}.
 
\begin{figure}[h!]
    \centering
    \includegraphics[width=0.8\textwidth]{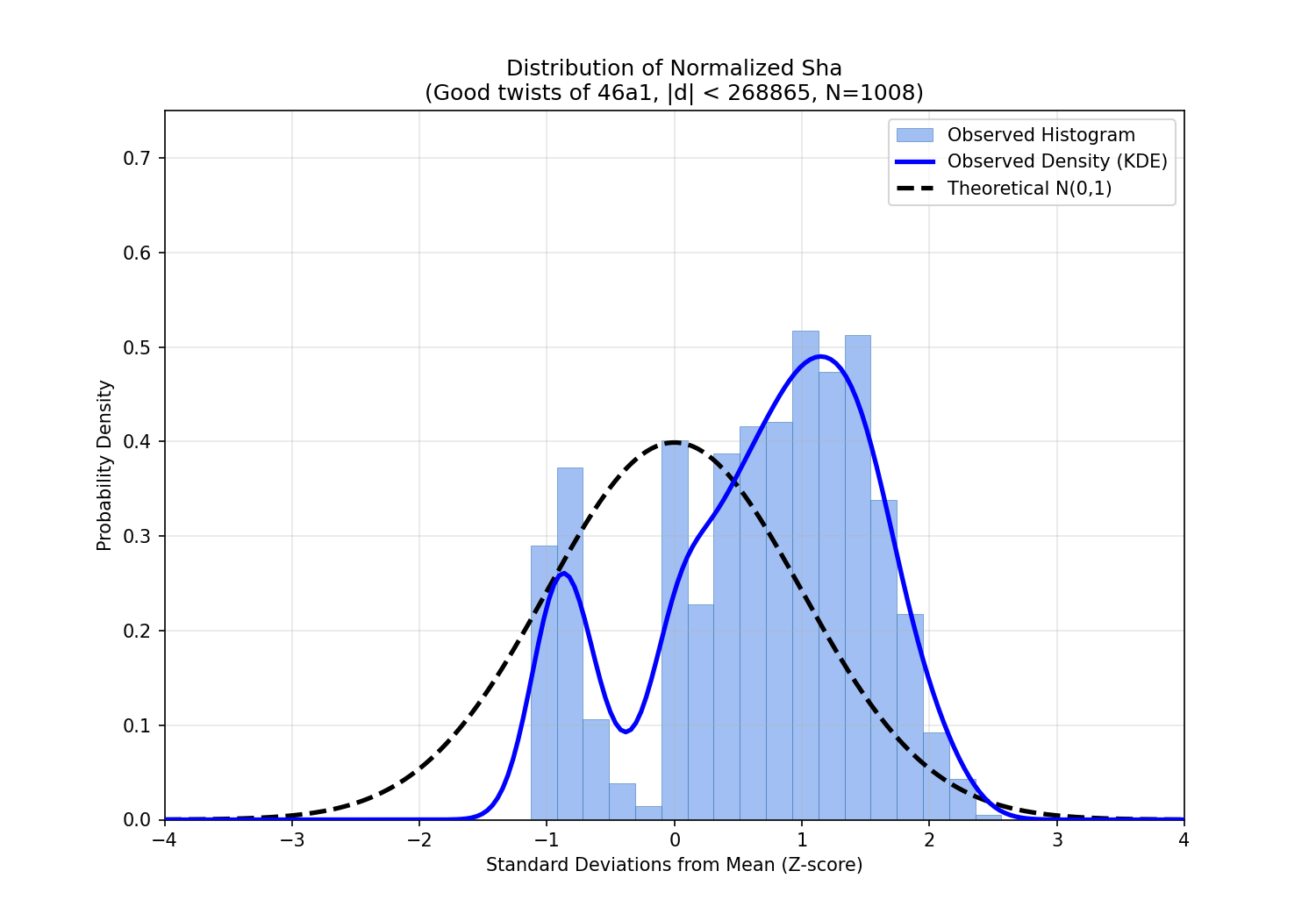}
    \caption{Distribution of normalized $Z_d$ for $1008$ of the unconditionally BSD-satisfying twists of \texttt{46a1}. The observed distribution (blue) is clearly shifted to the right of the generic Radziwiłł--Soundararajan prediction (black dashed). This shift reflects the bias introduced by restricting to real quadratic fields with split primes at the conductor. A time-lapse visualization of the convergence to $\mathcal{N}(0,1)$ is available at \href{https://radzi-sound-conjecture.netlify.app/46a1_restricted_nf50_20260115_014127_pdf}{this link}.}
    \label{fig:biased_distribution}
\end{figure}

The behaviour of the two distributions is clearly significantly different, with the main qualitative difference being that \Cref{fig:biased_distribution} has shifted to a bimodal shape. 

This anomaly need not be a failure of the Radziwiłł--Soundararajan heuristics, but rather a demonstration of their sensitivity to arithmetic progressions. The generic parameters $\mu(E)$ and $\sigma(E)$ are derived by averaging over the Galois group $G$ under the assumption that Frobenius elements are equidistributed as $d$ varies. However, our family $\mathcal{D}_{\mathrm{BSD}}$ explicitly restricts the factorization of $d$ (forcing $p \equiv 1 \pmod 4$), fixes the sign of the discriminant (forcing real fields), and fixes the quadratic character at the conductor (forcing $\left(\frac{d}{23}\right) = 1$). These restrictions alter the density of primes where our twists are supported, necessitating a correction to the probabilistic model. 

If we observed persistent deviations from Gaussian behavior in $|\Sha(E_d)|$ across unconstrained, randomly sampled families of quadratic twists, such as the baseline family shown in \Cref{fig:baseline_distribution}, this would provide more plausible evidence against the conjecture. However, we do not observe such a deviation in the baseline distribution.
We also note that the introduction of Radziwiłł--Soundararajan mentions that their method ``is flexible enough to allow the introduction of a sieve over the fundamental discriminants $d$''; this being the case, we could still expect the distribution in \Cref{fig:biased_distribution} to become approximately Gaussian (with different mean and variance) with more data, and this would be worthy of further investigation. 

\bibliographystyle{amsplain}
\bibliography{references.bib}{}

\end{document}